\documentclass[11pt]{article}
\usepackage[a4paper, total={7in, 9in}]{geometry}

\usepackage{lineno}
\usepackage{amssymb}
\usepackage{amsmath}
\usepackage{amsthm}
\usepackage{pifont}

\newtheorem{thm}{\bfseries Theorem}
\newtheorem{lem}[thm]{\bfseries Lemma}        
\newtheorem{prop}[thm]{\bfseries Proposition} 

\newtheorem{defn}[thm]{\bfseries Definition}

\def\set#1{\{#1\}}
\def\R{\mathbb R}

\begin{document}


\title{Nearest neighbor representations of Boolean functions}

\date{2020.~April}

\author
{
P\'eter Hajnal\thanks{
Supported by National Research, Development and Innovation
Office NKFIH Fund No. 129597, SNN-117879, 2018-1.2.1-NKP-2018-00004,
and SZTE, Smart Institute,
Hungary through grant TUDFO/47138-1/2019-ITM
of the Ministry for Innovation and Technology, Hungary.}\\
University of Szeged, Bolyai Institute\\
Aradi V\'ertan\'uk tere 1., Szeged, Hungary 6720
\and
Zhihao Liu\\
Brightloom, San Francisco, California
\and
Gy\"orgy Tur\'an\thanks{Supported by the National Research, Development and Innovation Office of Hungary through the Artificial Intelligence National Excellence Program (grant no.: 2018-1.2.1-NKP-2018-00008)}\\
University of Illinois at Chicago,\\
and Research Group on Artificial Intelligence of\\
the Hungarian Academy of Sciences at the University of Szeged\\
University of Szeged
}

\maketitle

\begin{abstract}
 A nearest neighbor representation of a Boolean function is a set
 of positive and negative prototypes in $\R^n$ such that the
 function has value 1 on an input iff the closest prototype is positive.
 For $k$-nearest neighbor representation the majority classification
 of the $k$ closest prototypes is considered. The nearest neighbor
 complexity of a Boolean function is the minimal number of
 prototypes needed to represent the function.
 We give several bounds for this measure. Separations are given
 between the cases when prototypes can be real or are required to be Boolean.
 The complexity of parity is determined exactly.
 An exponential lower bound is given for mod 2 inner product,
 and a linear lower bound is given for its $k$-nearest neighbor complexity.
 The results are proven using connections to other models
 such as polynomial threshold functions over $\{1, 2\}$.
 We also discuss some of the many open problems arising.
\end{abstract}

\section{Introduction}\label{intro}

A nearest neighbor representation of a classification of 
a set of points in $\R^n$ is given by a set of prototypes 
such that each point
belongs to the same class as the prototype closest to it.
More generally, for a $k$-nearest neighbor representation, 
the class containing
a point is determined by taking the most frequent 
class label among the $k$ closest prototypes.
Nearest neighbor representation is basic concept, 
much studied and used in computational geometry, 
machine learning,
pattern recognition and other areas. Case-based 
representation and case-based reasoning refer 
to the same idea but the notions of similarity 
used there are typically symbolic ones rather than geometric.

A general objective is to use as few prototypes as possible.
This leads to questions about the smallest number 
of prototypes representing a given classification.
We consider the special case of binary classifications 
of the $n$-dimensional hypercube.
A binary classification of the hypercube can be viewed 
as a Boolean function and therefore we use this terminology 
in the rest of the paper.

The question studied here is the following: 
what is the minimal number of prototypes needed 
to represent a Boolean function?
This quantity is referred to as the \emph{nearest neighbor complexity} 
of a Boolean function. The more 
general notion mentioned above is called the
$k$-\emph{nearest neighbor complexity} of a Boolean function. 
One can distinguish two types of nearest neighbor problems 
depending on whether the prototypes
must belong to $\{0, 1\}^n$, or can be arbitrary points in $\R^n$. 
In the former case we talk about 
\emph{Boolean nearest neighbor complexity}.

A basic example for the nearest neighbor representation 
of a Boolean function is the parity function. 
It follows directly from the definitions that 
in the Boolean nearest neighbor version 
every vertex has to be a prototype. On the other hand, 
in the general case the centers of the $n+1$ levels 
of the hypercube can serve as prototypes.
We show that this construction is optimal.

Linear threshold functions clearly can be 
represented with only two prototypes, but 
the Boolean prototype case is different. 
The majority function is an example 
with polynomial Boolean nearest neighbor complexity. 
On the other hand, the Boolean threshold function 
$TH_n^{\lfloor n/3 \rfloor}$ requires an 
exponential number of Boolean prototypes.

Regarding the nearest neighbor complexity of typical functions, 
it is shown that $O(2^n/n)$ prototypes are always sufficient, 
and almost all functions require $\Omega(2^{n/2}/n)$ prototypes. 
We also give an exponential lower bound for 
the nearest neighbor complexity of an explicit function: 
the mod 2 inner product function requires at least $2^{n/2}$ prototypes. 
For $k$-nearest neighbor complexity 
a linear lower bound is given for the same function.

Nearest neighbor complexity seems to be a natural 
complexity measure with many connections to other, 
well-studied complexity measures such as 
the complexity of threshold circuits, 
polynomial threshold functions and linear decision trees. 
In fact, the lower bounds for parity and mod 2 inner product 
turn out to follow directly from previous results on these measures. 
The parity lower bound is closely related to 
results of Basu \emph{et al.}~\cite{Basu2008}. 
The $\mod 2$ inner product lower bound is actually 
a special case of a result of Hansen and Podolskii~\cite{HansenP} 
(more details about this connection are given in the final section). 
Both \cite{Basu2008} and \cite{HansenP} discuss 
polynomial threshold functions over finite domains like $\{1, 2\}$ 
instead of $\{0, 1\}$.
The $k$-nearest neighbor lower bound follows from 
a result of Gr\"oger and Tur\'an~\cite{Groger1991} 
on linear decision trees (decision trees with linear function tests 
in the nodes). The relevance of these models for nearest neighbor problems 
is perhaps unexpected.

Nearest neighbor complexity and its relationship 
to other complexity measures raises many open problems. 
Some of these are formulated
in the final section of the paper\footnote{
This paper is a significantly updated version of a paper published with
the same title in the online, non-archival proceedings of
\emph{9th International Symposium on Artificial Intelligence
and Mathematics, ISAIM 2006}
(http://anytime.cs.umass.edu/aimath06/proceedings.html).
The present section \ref{open}
includes several new open problems, while some original questions
are solved in section ~\ref{sec:expl}.}.

\section{Related work}

General background on nearest neighbor methods and case based reasoning in machine learning, learning theory and computational geometry is given in Mitchell~\cite{Mitchelluj},
Hastie \emph{et al.}~\cite{Hastie2009}, Ben-David and Shalev-Shwartz~\cite{BenDavid2014} and Mulmuley~\cite{Mulmuleyuj}.  Recent work on these topics is discussed, for example, in
Luxburg and Bousquet~\cite{Luxburg04}, Klenk \emph{et al.}~\cite{Klenk11} and Anthony and Ratsaby~\cite{AnthonyR18}.

Nearest neighbors is an interesting topic in the context of deep learning as well. An important issue is interpretability: in several applications learned models need to be comprehensible for the user. Learning prototypes is considered to be interpretable in some contexts, and relevant for human cognition. Thus, for example, there are proposals to use deep learning to learn a set of prototypes (see, e.g., Li \emph{et al.}~\cite{RudinNN} and Zoran \emph{et al.}~\cite{Zoran2017}).

Wilfong~\cite{Wilfong1992} considers the problem of finding a minimal set of prototypes \emph{within} a set of prototypes for a given classification of a set.
The set to be found could be required to give the correct classification of the whole set, or only the original set of prototypes (in this case it is called a consistent subset). In the latter case the problem is the same as the one discussed in this paper. The problem studied is the computational complexity of the problem. Recent results include
Banerjee~\cite{Banerjee2018}, Gottlieb~\cite{Gottlieb2018} and Biniaz \emph{et al.}~\cite{Biniaz19}.

The paper closest to our work is
Salzberg \emph{et al.}~\cite{Salzberg1995}, considering nearest neighbor complexity for two-dimensional geometric concepts. Their motivation is that a minimal set of prototypes can be used by a helpful teacher to teach a concept, e.g., a convex polygon. This can also be viewed as the best case data set for learning with the nearest neighbor algorithm. They also mention other questions such as the stability of the set of prototypes, i.e., how adding more prototypes influences the representation.
The Boolean complexity background is discussed in Jukna~\cite{Jukna2012}. As noted in the introduction, we use results from Hansen and Podolskii~\cite{HansenP}, Basu \emph{et al.}~\cite{Basu2008} and Gr\"oger and Tur\'an~\cite{Groger1991}.

Globig and Lange~\cite{Globig1996} discuss nearest neighbor complexity for Boolean functions. They consider nearest neighbor representations for a \emph{class} of Boolean functions
for \emph{some} distance function (not necessarily a metric) which allows for a polynomial size representation of functions belonging to that class. Thus constructing an appropriate distance function becomes part of the problem. Satoh~\cite{Satoh98uj} looks at the Boolean case where the (dis)similarity of two Boolean vectors is the \emph{set} of coordinates where they differ, and $a$ is closer to $b$ than to $c$ if the difference of $a$ and $b$ is a \emph{subset} of the difference of $a$ and $c$ (thus this is a partial ordering).
Bengio \emph{et al.}~\cite{Bengio05} consider related questions for radial basis functions or Gaussian kernels.

\section{Preliminaries}

The Euclidean distance in $\R^n$ (resp., the Hamming distance in $\set{0, 1}^n$) is denoted by $d(a, b)$ (resp., $d_H (a, b)$).
For $a, b \in\set{0, 1}^n$ it holds that $d(a, b) = \sqrt{d_H (a, b)}$.
The componentwise partial order on $\set{0,1}^n$ is denoted by $a \leq b$.
The all-0 (resp., all-1) vector is denoted by ${\bf 0}$ (resp., ${\bf 1}$).
The weight of $a = (a_1 ,\ldots, a_n)\in\set{0,1}^n$ is $\sum_{i=1}^n a_i$.
If $a \leq b$ then we also say that $a$ is covered by $b$.
For a vector $a = (a_1 ,\ldots, a_n)\in\set{0,1}^n$, we write $a^{(i)}$ for
the vector obtained from $a$ by switching its $i$'th component, and we write $|a|$ for its weight, i.e., the number of its $1$ components.
Switching a component $1$ in $a$ to $0$ we get a lower neighbor of $a$.
Let $f:\set{0,1}^n \to \set{0,1}$ be a Boolean function.
Truth assignments $a \in \{0,1\}^n$ with $f(a)=1$ (resp., $f(a)=0$) are called
positive (resp., negative).

\begin{defn}
A \emph{nearest neighbor (NN)} representation of a Boolean function $f$ is a
pair of disjoint subsets $(P, N)$ of $\R^n$ such that for every $a \in \set{0, 1}^n$
\begin{itemize}
\item[$\bullet$]
if $a$ is positive then there exists $b\in P$ such that for every $c\in N$ it holds that
$d(a, b) < d(a, c)$,
\item[$\bullet$]
if $a$ is negative then there exists $b\in N$ such that for every $c\in P$ it holds that
$d(a, b) < d(a, c)$.
\end{itemize}

The points in $P$ (resp., $N$) are called positive (resp., negative) \emph{prototypes}.
The size of the representation is $|P\cup N|$.
The nearest neighbor complexity, $NN(f)$, of $f$ is the minimum of the sizes of the representations of $f$.
A nearest neighbor representation is \emph{Boolean} if $P\cup N\subseteq\set{0, 1}^n$, i.e., if the prototypes are Boolean vectors.
The minimum of the sizes of the Boolean nearest neighbor representations is denoted by $BNN(f)$.
\end{defn}

If $f$ is non-constant then $NN(f) \ge 2$ as at least one positive and negative prototype is needed.
 For every $n$-variable Boolean function it holds that
 \begin{equation} \label{eq:alap}
NN(f)\leq BNN(f)\leq 2^n.
\end{equation}
Here the first inequality follows from the definitions and the second inequality follows from using all truth assignments as prototypes.

Nearest neighbor representations can be generalized to $k$-nearest neighbor representations.

\begin{defn}
A $k$-nearest neighbor ($k\text{-}NN$) representation of $f$ is a pair of disjoint subsets $(P, N)$ of $\R^n$,
such that for every $a \in \set{0, 1}^n$ it holds that

\begin{itemize}
\item[$\bullet$]
$a$ is positive iff at least $\frac{k}{2}$ of the $k$ points in $P \cup N$ closest to $a$ belong to $P$.
\end{itemize}

It is assumed that for every $a$, the $k$ smallest
distances of $a$ from the prototypes are all smaller than the other $|P\cup N|-k$ distances from the prototypes.
Thus the case $k = 1$ is the same as the nearest neighbor representation.
The size of the representation is again $|P\cup N|$.
The $k$-nearest neighbor complexity, $k\text{-}NN(f)$,
of $f$ is the minimum of the sizes of the $k$-nearest neighbor representations of $f$.
\end{defn}

\section{Boolean nearest neighbors} \label{sec:boo}

In this section we discuss some basic examples showing that there can be exponential gaps in the inequalities in (\ref{eq:alap}).
For the parity function there is an exponential gap in the first inequality, and for the majority functions there is an exponential gap in the second inequality.

A Boolean function is \emph{symmetric} if its value depends only on the weight of its input. A symmetric function $f$ can be specified by a set $I_f \subseteq \{0, \ldots, n\}$ such that $f(a) = 1$ iff $|a| \in I_f$.

\begin{prop} \label{pr:szim}
\begin{itemize}
\item[a)]
For every $n$-variable symmetric function $f$ it holds that $NN (f )\leq n + 1$.
\item[b)]
$BN N (x_1\oplus x_2\oplus\ldots\oplus x_n)= 2^n$ .
\end{itemize}
\end{prop}

\begin{proof}
For part \emph{a)}, consider prototypes $p_\ell= (\frac{\ell}{n},\ldots,\frac{\ell}{n})$, for $\ell= 0,\ldots, n$.
If $a \in \set{0, 1}^n$ has weight $w$ then $d(a, p_w ) < d(a, p_\ell)$ for every $\ell\not= w$,
as the hyperplane $\sum_{i=1}^n x_i = w$ is perpendicular to the line $(t, \ldots, t)$.
Thus $(P, N)$ with $P = \set{p_\ell : \ell \in I_f}$ and $N = \set{p_\ell : \ell \not\in I_f}$
is a nearest neighbor representation of size $n + 1$.

For part \emph{b)}, consider a Boolean nearest neighbor representation of the parity function and
let $p$ be a positive prototype.  If $a$ is a neighbor of $p$ then $a$ is
negative, but there is a positive prototype at distance $1$ from $a$.
Hence $a$ must itself be a negative prototype. Repeating this argument
it follows that every point is a prototype.
\end{proof}

In Section~\ref{sec:expl} we show that the upper bound of part \emph{a)} is sharp for the parity function.

A Boolean function $f$ is a \emph{threshold} function if there are weights
$w_1,\ldots, w_n \in\R$ and a threshold $t\in\R$
such that for every $x\in\set{0,1}^n$
it holds that $f (x) = 1$ iff $w_1 x_1 + \ldots + w_n x_n\geq  t$. The
special case when $w_1 = \ldots = w_n = 1$ is denoted by $TH_n^t$ . In
particular, when $t =\frac{n}{2}$, we get
the $n$-variable majority function $MAJ_n (x)$.

\begin{thm}
\begin{itemize}
\item[a)]
For every threshold function $f$ it holds that $NN (f ) = 2$.
\item[b)]
If $n$ is odd then $BNN (MAJ_n ) = 2$ and
if $n$ is even then $BNN (MAJ_n )\leq\frac{n}{2}+2$.
\item[c)]
$BNN\left(TH_n^{\lfloor n/3 \rfloor}\right)=2^{\Omega(n)}$.
\end{itemize}
\end{thm}

\begin{proof}
Part \emph{a)} follows by taking a single positive, resp. negative,
prototype, on a line perpendicular to the hyperplane defining the
threshold function, at equal distances from the hyperplane. (It may be assumed \emph{w.l.o.g.} that the hyperplane does not contain any point from $\{0, 1\}^n$.)

Part \emph{b)} is obtained for odd $n$ by taking the all $0$, resp. all $1$,
vectors as prototypes. In the even case let the all $0$ vector be the
single negative prototype, and select arbitrary $(n/2) + 1$
truth assignments of weight $n-1$ as positive prototypes.
Then every truth assignment $a$ of weight $n/2$ shares a $0$ component with some positive prototype.
Their Hamming distance is $(n/2)-1$, and so this prototype is closer to $a$ than the all $0$ vector.
It is easy to check that if $x$ has weight different from $n/2$,
then the prototype closest to it has the right label.

For part \emph{c)}, let $t = \lfloor n/3 \rfloor$ and consider Boolean prototypes $P, N \subseteq \set{0, 1}^n$ for $TH_n^t$.
Let $a$ be a truth assignment of weight $t$, and $p$ be a positive prototype closest to $a$. We claim that $a \leq p$.
Otherwise assume that $a_i = 1$, $p_i = 0$ and consider the negative truth assignment $b = a^{(i)}$.
Let a negative prototype closest to $b$ be $q$.  Then
$d_H (a, p) = d_H (b, p) + 1 > d_H (b,q) + 1$.
On the other hand
$d_H (a, p) < d_H (a, q) \leq d_H (b, q) + 1$, a
contradiction.

It follows similarly that if $b$ is a truth assignment of weight $t-1$ and $q$ is a
negative prototype closest to $b$ then $q \leq b$. This implies that for
every truth assignment $a$ of weight $t$ there is a negative prototype $q$ such that
$q\leq a$ (a prototype closest to a lower neighbor of $a$ will have this
property). Thus for every truth assignment $a$ of weight $t$ it holds that
$d_H(a,q) \leq t$ for some negative prototype $q$.
This means that if $p$ is a positive prototype closest to $a$ then $d_H(a,p) < t$ and so $|p| < 2t$.

Consider now the set of truth assignments of weight $t$. Each is covered by a
positive prototype of weight less than $2t$. Each such positive
prototype can cover at most $\binom{2t}{t}$ truth assignments of weight $t$.
Hence we need at least
\[
\frac{\binom{n}{t}}{\binom{2t}{t}}=2^{\Omega(n)}
\]
positive prototypes.
\end{proof}

The argument of part c) generalizes to every function $TH^t_n$, where
$|t-\frac{n}{2}|\geq\delta n$ for any fixed $\delta>0$.

\section{Largest nearest neighbor complexity}

In this section we consider Shannon-Lupanov type bounds.
The first bound shows that the upper bound of (1) for nearest neighbor complexity can be improved asymptotically by a factor of $1/n$.

\begin{thm} \label{th:shau}
For every n-variable Boolean function it holds that
\[ 
NN(f)\leq(1 + o(1))\frac{2^{n+2}}{n}. 
\]
\end{thm}

\begin{proof}
A set $B_c\subseteq\set{0,1}^n$ is a ball of radius one with center $c$ if it
consists of $c \in \set{0,1}^n$ (the center of the ball) and all its neighbors.
A set $S_c\subseteq\set{0,1}^n$ is a sphere of radius one  with center $c$ if it consists of all the neighbors of $c\in\set{0,1}^n$.

\begin{lem}
Let $A$ be a subset of a sphere $S_c$ of radius one with $|A|=\ell\geq 3$, and let
$c_A=\frac{1}{|A|}\sum_{x\in A} x$ be the centroid of $A$. Then
\begin{itemize}
\item[a)]
$d(c_A, a) < 1$ for every $a\in A$,
\item[b)]
$d(c_A, a) \geq 1$ for every $a \in \{0,1\}^n \setminus (A \cup \{c\})$.
\end{itemize}
\end{lem}

\begin{proof}
Assume \emph{w.l.o.g.} that $S$ has center ${\bf 0}$, and $A$ consists of the first $\ell$ unit vectors.
Then
$c_A = (\frac{1}{\ell},\ldots,\frac{1}{\ell},0,\ldots,0)$,
where the first $\ell$ coordinates are nonzero.
If $a\in A$ then
\[ 
d^2(c_A, a) = \left(\frac{\ell-1}{\ell}\right)^2 + 
(\ell-1)\left(\frac{1}{\ell}\right)^2 = 
\frac{\ell-1}{\ell} < 1.
\]
If $a \in \{0,1\}^n \setminus (A \cup \{c\})$ then if $a$ has a $1$ component in the last $n-\ell$ coordinates then $d(c_A,a)\geq 1$.
Otherwise $a$ has at least two $1$'s in the first $\ell$ coordinates and so as $\ell\geq 3$ it holds that
\[
d^2(c_A,a)\geq2\left(\frac{\ell-1}{\ell}\right)^2 
+(\ell-2)\left(\frac{1}{\ell}\right)^2
=2-\frac{3}{\ell}\geq 1.
\]
\end{proof}

Partition $\set{0,1}^n$ into subsets $A_1,\ldots,A_s$ such that each $A_i$ is a subset of some ball $B_i$ of radius
one with center $c_i$, and let $A_i^1$
(resp.~$A_i^0$) be the set of points $x \not= c_i$ in $A_i$
with $f(x) = 1$ (resp.~f(x) = 0). In each $A_i$ pick the following prototypes:
\begin{itemize}
\item[$\bullet$]
if $|A_i^1|\geq 3$ then let $c_{A_i^1}$
be a positive prototype, otherwise let the element(s) of $A_i^1$ be positive prototypes,
\item[$\bullet$]
if $|A_i^0|\geq 3$ then let $c_{A_i^0}$
be a negative prototype, otherwise let the element(s) of $A_i^0$ be negative prototypes,
\item[$\bullet$]
if the center $c_i\in A_i$ then let $c_i$ be a prototype with label $f(c_i)$.
\end{itemize}
The correctness of this set of prototypes follows from Lemma 4. The theorem then follows from
the result that $\set{0,1}^n$ can be covered with
$(1 + o(1))\frac{2^n}{n}$ balls of radius one (Kabatyansky and Panchenko~\cite{Kabat1988}), generalizing Hamming codes.
\end{proof}

As the next result shows, almost all $n$-variable functions have exponential complexity.

\begin{thm} \label{th:shal}
For almost all n-variable Boolean functions
\[
NN(f) >\frac{2^{n/2}}{n}.
\]
\end{thm}

\begin{proof}
Consider a set of prototypes $p_1,\ldots,p_m$ for some function $f$.
By slightly perturbing the points if necessary, it may be assumed \emph{w.l.o.g.} that
$d(x,p_i)\not=d(x,p_j)$ for every $x\in\set{0,1}^n$ and $1\leq i<j\leq m$.
The distances $d(x,p_i)$ and $d(x,p_j)$ can be compared by considering the
hyperplane $H_{p_i,p_j}$ going through the midpoint of the segment $(p_i,p_j)$, perpendicular to the segment,
and determining on which side of the hyperplane $x$ lies.
If for another set of prototypes $q_1,\ldots,q_m$ (again, without ties), the hyperplanes $H_{q_i,q_j}$
determine the same dichotomy of $\set{0,1}^n$ as $H_{p_i,p_j}$ for every $1\leq i<j\leq m$,
then $q_1,\ldots,q_m$ are prototypes for the same function $f$.

Hyperplanes can realize at most $2^{n^2}$ dichotomies of $x$ and thus $m$
prototypes can realize at most
\begin{equation} \label{eq:bino}
2^{n^2\binom{m}{2}}
\end{equation}
$n$-variable Boolean functions.
If a function can be realized with less than $m$ prototypes then it can also be realized with $m$ prototypes.
A direct calculation shows that for $m = \frac{2^{n/2}}{n}$ the quantity (\ref{eq:bino}) is $o\left(2^{2^n}\right)$.
\end{proof}

One actually gets the same bound for $k$-nearest neighbors as well.
The only difference in the proof is that a set of $m$ prototypes can represent $m$ different functions for different values of $k$.
Thus the upper bound (2) has to be multiplied by $m$, but the same bound remains valid.

\begin{thm}
For almost all $n$-variable Boolean functions $f$ it holds that for every $k$
\[ 
k\text{-}NN(f) > \frac{2^{n/2}}{n}.
\qed \]
\end{thm}

\section{Lower bounds for the nearest neighbor complexity of explicit functions} \label{sec:expl}

In this section we give an exponential lower bound for the nearest neighbor complexity of the mod 2 inner product function and show that the upper bound of Proposition~\ref{pr:szim} for the parity function is sharp. Both lower bounds are based on sign-representations of Boolean functions by polynomials.

A multivariate polynomial $p(x_1, \ldots, x_n)$ is a \emph{sign-representation} of a Boolean function $f(x_1, \ldots, x_n)$ if for every $x = (x_1, \ldots, x_n) \in \{0, 1\}^n$ it holds that $p(x) \ge 0$ iff $f(x) = 1$.

As mentioned in the introduction, Boolean functions can be represented with arguments having two possible values different from 0 and 1.
The Fourier representation of Boolean functions uses values $\pm 1$, and $(-1)^x$ gives a bijection between $\{0, 1\}$ and $\{1, -1\}$. For the $\{1, 2\}$-representation of Boolean functions it holds similarly that $2^x$ gives a bijection between $\{0, 1\}$ and $\{1, 2\}$. The corresponding change of variables is denoted by $\tilde{x_i} = 2^{x_i}$, and the function obtained after the transformation is denoted by $\tilde{f}$.
A monomial in the $\{1, 2\}$-representation can thus be written as
\[ 
\tilde{x}^{a_1} \ldots \tilde{x}_n^{a_n} = 
2^{a_1 x_1 + \ldots a_n x_n}. 
\]
A multivariate polynomial $p(\tilde{x}_1, \ldots, \tilde{x}_n)$ is a \emph{$\{1,2\}$-sign-representation} of a Boolean function $f(x_1, \ldots, x_n)$ if for every $\tilde{x} = (\tilde{x}_1, \ldots, \tilde{x}_n) \in \{1, 2\}^n$ it holds that $p(\tilde{x}) \ge 0$ iff $\tilde{f}(\tilde{x}) = f(x) = 1$.

Both lower bounds follow by relating nearest neighbor representations to sign-representations over $\{1,2\}$, and using lower bounds for such sign-representations.

\begin{lem} \label{le:ford}
If a Boolean function has a nearest neighbor representation with $m$ prototypes then it has a sign-representation over $\{1, 2\}$ having
$m$ terms.
\end{lem}

\begin{proof} Let $f : \{0, 1\}^{n} \to \{0, 1\}$ be an $n$-variable Boolean function and
let $(P, N)$ be a nearest neighbor representation of $f$, with $P = \{a_1, \ldots, a_p\}$ and $N = \{b_1, \ldots, b_q\}$, where $m = p + q$.
We assume \emph{w.l.o.g.} that every coordinate is rational.

As $x_i^2 = x_i$, squared distances can be written as linear forms
\[ 
d^2(x, a_k) = ||a_k||^2 + 
\sum_{i=1}^{n} (1 - 2 \, a_{k,i}) \cdot x_i 
\]
for $P$ and similarly for $N$. Consider the expressions
\[ 
A \cdot\left(B \cdot \left(1 + \sum_{i=1}^n x_i\right) - d^2(x, a_k)\right),
 \,\,\,\, 
 A \cdot \left(B \cdot \left(1+\sum_{i=1}^n x_i\right) - d^2(x, b_\ell)\right),
 \]
where $A, B$ are integers and the squared distances are written as linear forms.
Choose $B$ to be sufficiently large for making all coefficients 
in the expressions 
$(B \cdot (1 + \sum_{i=1}^n x_i) - d^2(x, a_k))$ 
and $(B \cdot (1 + \sum_{i=1}^n x_i) - d^2(x, b_\ell))$ 
become nonnegative. Let $M$ be 
the least common multiple of all denominators, 
and let $A = \lceil \log (p + q) \rceil \cdot M$.
Denote the coefficients of the linear forms 
thus obtained by $\alpha_{k,i}, \beta_{\ell,i}$ and 
consider the expression
\[ 
\sum_{k=1}^p   2^{(\sum_{i=0}^n \alpha_{k,i} x_i)} - 
\sum_{\ell=1}^q 2^{(\sum_{i=0}^n  \beta_{\ell,i} x_i)}.
\]

As all the coefficients are nonnegative integers, 
the expression can be written as
\begin{eqnarray}
\sum_{k=1}^p 2^{\alpha_{k,0}} \cdot 
\left(\prod_{i=1}^n (2^{x_i})^{\alpha_{k,i}}\right) 
- \sum_{\ell=1}^q 2^{\beta_{\ell,0}} \cdot 
\left(\prod_{i=1}^n (2^{x_i})^{\beta_{\ell,i}}\right) \nonumber \\
= \sum_{k=1}^p 2^{\alpha_{k,0}} \cdot 
\left(\prod_{i=1}^n \tilde{x_i}^{\alpha_{k,i}}\right) - 
\sum_{\ell=1}^q 2^{\beta_{\ell,0}} \cdot 
\left(\prod_{i=1}^n \tilde{x_i}^{\beta_{\ell,i}}\right). \nonumber
\end{eqnarray}
We claim that this polynomial is a $\{1, 2\}$-sign representation of $f$. If $f(x) = 1$ then for some $k$ it holds that
\[ 
M \cdot (B \cdot (1 + \sum_{i=1}^n x_i) - d^2(x, a_k)) 
\ge 
M \cdot (B \cdot (1 + \sum_{i=1}^n x_i) - d^2(x, b_\ell)) + 1 
\]
for every $\ell$, and so
\[ 
A \cdot (B \cdot (1 + \sum_{i=1}^n x_i) - d^2(x, a_k)) > 
A \cdot (B \cdot (1 + \sum_{i=1}^n x_i) - d^2(x, b_\ell)) + 
\lceil \log (p + q) \rceil 
\]
for every $\ell$. Thus the polynomial term corresponding to $k$  is larger than the sum of the terms corresponding to $N$. The case $f(x) = 0$ follows analogously.
\end{proof}

The mod 2 inner product function of $2n$ variables is defined by
\[ 
IP_n(x_1,\ldots,x_n,y_1,\ldots,y_n) = 
(x_1\wedge y_1)\oplus\ldots\oplus(x_n\wedge y_n). 
\]

\begin{thm}\label{th:also9} 
\emph{a)} $NN(IP_n) \geq 2^{n/2}$.

\emph{b)} $NN(x_1 \oplus \dots \oplus x_n) \geq n + 1$.
\end{thm}

\begin{proof} 
The lower bounds follow by combining Lemma~\ref{le:ford} 
with lower bounds of Hansen and Podolskii~\cite{HansenP}, 
resp., Basu \emph{et al.}~\cite{Basu2008} 
for such sign-representations of inner product. 
For completeness we include proofs of these lower bounds.

\begin{lem} (Hansen, Podolskii~\cite{HansenP}) \label{le:hap}
Every sign-representation of $IP_n$ over $\{1, 2\}$ has 
at least $2^{n/2}$ terms.
\end{lem}

\begin{proof} Consider an $m$-term polynomial 
$p(\tilde{x}_1, \ldots, \tilde{x}_n, \tilde{y}_1, \ldots, \tilde{y}_n)$ 
which is a sign-representation over $\{1,2\}$ of $IP_n$.
Let
\[ 
c \cdot \prod_{i=1}^n \tilde{x}_i^{a_i}\cdot\prod_{i=1}^n \tilde{y}_i^{b_i} 
= c \cdot 2^{\sum_{i=1}^n a_i x_i} \cdot 2^{\sum_{i=1}^n b_i y_i} 
\]
be a monomial in $p$. The $2^n \times 2^n$ matrix with rows labeled $x$,
columns labeled $y$ and entries given by this monomial is 
a rank-1 matrix obtained by multiplying the column vector 
consisting of the $x$-parts of the monomial 
with the row vector consisting of the $y$-parts. 
Thus the $2^n \times 2^n$ matrix representing $IP_n$ is sign-represented 
by a linear combination of these matrices. 
F\"orster's theorem~\cite{Forster2002} is that then $m\ge 2^{n/2}$.
\end{proof}

\begin{lem} (Basu \emph{et al.}~\cite{Basu2008})
Every sign-representation of $n$-variable parity over $\{1, 2\}$ 
has at least $n + 1$ terms.
\end{lem}

\begin{proof} The claim is obvious for $n = 1$.
Let
\[ 
p(x_1,\ldots,x_n) = \sum_{i=1}^k 
p_i(x_1,\ldots,x_{n-1}) \cdot x_n^{d_i} 
\]
be a sign-representation of parity over $\{1, 2\}$, where $0 \le d_1 < d_2 < \ldots$. Writing $x_1, \ldots, x_{n-1}$ as $x'$

\begin{eqnarray} p(x', 1) &=& p_1(x') + p_2(x') +
\sum_{i=3}^k p_i(x') \nonumber \\
p(x', 2) &=& p_1(x') \cdot 2^{d_1} + p_2(x') \cdot 2^{d_2} +
\sum_{i=3}^k p_i(x') \cdot 2^{d_i}. \nonumber
\end{eqnarray}

Let $c = 1/(2^{d_2} - 2^{d_1})$. Then
\[ 
c \cdot (- p(x', 1) \cdot 2^{d_1} + p(x', 2)) =   
p_2(x')  + c \, \sum_{i=3}^k (2^{d_i} - 2^{d_2}) \cdot p_i(x') 
\]


As $(x', 1)$ and $(x', 2)$ have different parities, 
$p(x', 1)$ and $p(x', 2)$ have different signs. 
Therefore the polynomial on the right-hand side is 
a sign-representation of the parity of $n-1$ variables, 
and by induction it has at least $n$ terms. This implies that
the polynomial
\[ 
\sum_{i=2}^k p_i(x_1, \ldots, x_{n-1}) \cdot x_n^{d_i}  
\]
has at least $n$ terms as well. As 
$p_1(x_1, \ldots, x_{n-1}) \cdot x_n^{d_1}$ 
has at least one additional term, the lemma follows.
\end{proof}
\end{proof}

\section{A lower bound for $k$-nearest neighbor complexity}

In this section we


\bibliographystyle{abbrv}

 give a linear lower bound 
for the $k$-nearest neighbor complexity of mod 2 inner product.

A linear decision tree over the variables $x_1,\ldots,x_n$ 
is a binary tree, where each inner node is labeled
by a linear test of the form $w_1x+1 +\ldots+w_nx_n : t$, 
for some $w_1,\ldots,w_n, t \in R$, the edges leaving the
node are labelled $\leq$ and $>$, and the leaves are labeled $0$ and $1$.
For an input vector $x\in\set{0,1}^n$ 
the function value computed by the tree is 
the label of the leaf reached by following the path
corresponding to the results of the tests for $x$.
The linear decision tree complexity, $LDT(f)$, 
of a function $f$ is the minimum of 
the depths of linear decision trees computing $f$.

\begin{thm} \label{th:knear} For every $k$ it holds that
\[
k\text{-}NN(IP_n)\ge\frac{n}{6 + o(1)}.
\]
\end{thm}

\begin{proof}
We first formulate a general connection between $k$-nearest neighbor complexity
and the complexity of computing a function by linear decision trees.

\begin{lem}
For every $k$ and every Boolean function f it holds that $LDT(f) \le (3 + o(1)) \cdot {k\text{-}NN(f)}$.
\end{lem}

\begin{proof}
Consider a set of prototypes $p_1, \ldots, p_m$ for $f$.
The classification of  $x \in \{0, 1\}^n$ can be determined by finding the $k$ smallest squared distances $d^2(x, p_i)$ and checking the classifications of the prototypes involved.
Two squared distances can be compared using a linear test. The $k$'th smallest squared distance can be found using a selection algorithm performing $(3 + o(1)) m$ comparisons in the worst case. The algorithm also gives the $k$ smallest squared distances, and so it can be used to determine $f(x)$.
\end{proof}

In view of the lemma, the theorem is implied by the following lower bound of Gr\"oger and Tur\'an [4].

\begin{lem}
$LDT(IP_n)\geq\frac{n}{2}$.\qed
\end{lem}

\begin{proof}
Consider a linear decision tree computing $IP_n$, and let
\[ 
\sum_{i=1}^n a_i x_i +  \sum_{i=1}^n b_i y_i \ge t 
\]
be the linear test at the root. Rearrange the rows and columns of the $2^n \times 2^n$ matrix of $IP_n$ according to the value of the linear forms for the $x$, resp., the $y$ parts, and consider the input vector $(x^*, y^*)$ corresponding to the middle row and column. The depending on the outcome of the test for this vector, either all inputs in the top left or inputs in the bottom right quadrant all have the same outcome. Follow the corresponding edge in the tree and repeat. At the leaf arrived at the corresponding submatrix is constant.
However, the largest constant rectangle has size $2^{n/2} \times 2^{n/2}$  (see, e.g., Jukna~\cite{Jukna2012}), so the path leading to the leaf has length at least $n/2$.
\end{proof}


\end{proof}

The lower bound also applies 
to the \emph{weighted} version of $k$-nearest neighbor, 
where the classification is determined by a weighted sum 
of the classifications of the $k$ nearest neighbors, 
giving larger weight to closer ones.

\section{Remarks and open problems}\label{open}

In this section let $NN$ denote the class of Boolean functions 
with polynomial nearest neighbor complexity.
Hansen and Podolskii~\cite{HansenP} consider 
the class \emph{max-plus PTF} of Boolean functions $f$ 
defined by two polynomial size sets of linear functions
$L_i(x), i = 1, \ldots, p$ and $M_j(x), j = 1, \ldots, q$ such that
$f(x) = 1$ iff $\max_i L_i(x) \ge \max_j M_j(x)$.
It holds that $NN \subseteq max\text{-}plus PTF$, 
so Theorem~\ref{th:also9} follows from the $2^{n/2}$ lower bound for the
\emph{max-plus PTF} complexity of $\mod 2$ inner product 
in ~\cite{HansenP}, which is proved similarly to Lemma~\ref{le:hap}.
The class \emph{max-plus PTF} seems more general 
as it also allows for an additive constant weighting 
when evaluating distance from prototypes.
Showing that the containment $NN \subseteq max\text{-}plus PTF$ 
is proper would be interesting 
as it would require a different lower bound argument for nearest neighbors.

It is also noted in~\cite{HansenP} that 
$max\text{-}plus PTF \subseteq OR \circ AND \circ THR$. 
Here $OR \circ AND \circ THR$ is
the class of polynomial size depth-3 threshold circuits 
with an $OR$ gate at the top, $AND$ gates in the middle 
and $THR$ gates at the bottom.
Thus $NN$ is also contained in this class.
$OR \circ AND \circ THR$ circuits have a simple geometric interpretation: they correspond to a separation of the true and false points by a union of polyhedra.
Recently Murray and Williams~\cite{MuWi18} proved superpolynomial lower bounds for the larger class $ACC \circ THR$. Proving stronger bounds, even for $OR \circ AND \circ THR$, is an interesting challenge.

Proving bounds for $k$-nearest neighbor complexity is largely open. For example, it would be interesting to prove superpolynomial lower bounds (improving the lower bound of Theorem~\ref{th:knear}), and to consider the role of the parameter $k$. Hansen and Podolskii~\cite{HansenP} show that systems of \emph{max-plus PTF}s are \emph{equivalent} to $OR \circ AND \circ THR$. This observation gives additional motivation for considering $k$-nearest neighbor representations, as those also represent something like an additional level above nearest neighbor representations.

The gap between the upper bound of Theorem~\ref{th:shau} and the lower bound of Theorem~\ref{th:shal} should
be narrowed. Here the lower bound is proved using hyperplane counting, and the upper bound is proved using coding.
As far as we know there is a similar gap in Shannon - Lupanov type bounds for depth-2 threshold functions ($2^n/n^2$ versus $2^n/n$), where the bounds are also shown by similar hyperplane counting, resp., coding arguments (see, e.g., Spielman~\cite{Spielman92}).


Studying nearest neighbor complexity and $k$-nearest neighbor complexity for other metrics, such as the Manhattan metric or the metrics over $\{0, 1\}^n$ listed in Deza and Deza~\cite{Deza6}, would also be interesting. In particular, the lower bound of Theorem~\ref{th:also9} \emph{a)} applies to the Manhattan metric, and, more generally, to any distance where the contribution of the two halves of the input are additively separated. The lower bound of Theorem~\ref{th:also9} \emph{b)} applies to Gaussian kernels, and thus gives a matching lower bound for the upper bound of Remark 4.8 in Bengio \emph{et al.}~\cite{Bengio05}.

The following, even more general consideration is related to topics such as case-based reasoning and metric learning. Efficient algorithms for computing Boolean functions could be obtained by providing a representation with few prototypes with respect to some metric (or distance), which may be designed specifically for a class of functions. Examples of this approach are given in Globig and Lange~\cite{Globig1996}. The efficient computability of the distance function could also be considered.

\bigskip\noindent{\bf Acknowledgement}
We would like to thank Simon Kasif for suggesting the problem discussed in
this paper, and G\'abor Tardos for pointing out an improvement in Theorem~\ref{th:knear}.



\bibliographystyle{abbrv}


\end{document}